\documentclass[11pt]{amsart}
\usepackage{amssymb, amsmath}

\title[Uniformization of polynomial dynamics]{Arboreal Galois representations and  uniformization of polynomial dynamics}
\author{Patrick Ingram}
\address{Department of Mathematics, Colorado State University}
\email{pingram@math.colostate.edu}
\date{\today}

\newcommand{\QQ}{\mathbb{Q}}
\newcommand{\ZZ}{\mathbb{Z}}
\newcommand{\CC}{\mathbb{C}}
\newcommand{\RR}{\mathbb{R}}

\newcommand{\PP}{\mathbb{P}}

\renewcommand{\AA}{\mathbb{A}}

\newcommand{\basin}{\mathcal{B}}

\newcommand{\Gal}{\operatorname{Gal}}

\newcommand{\GL}{\operatorname{GL}}
\newcommand{\Aut}{\operatorname{Aut}}
\newcommand{\ord}{\operatorname{ord}}
\newcommand{\MOD}[1]{~(\textup{mod}~#1)}

\newcommand{\pf}{\mathfrak{p}}
\newcommand{\mf}{\mathfrak{m}}

\renewcommand{\epsilon}{\varepsilon}

\newcommand{\llbracket}{[\hspace{-1.5pt}[}
\newcommand{\rrbracket}{]\hspace{-1.5pt}]}

\newtheorem{theorem}{Theorem}

\newtheorem{lemma}[theorem]{Lemma}
\newtheorem{corollary}[theorem]{Corollary}

\theoremstyle{definition}

\begin{document}
\begin{abstract}
Given a polynomial $f$ defined over a complete local field, we construct a biholomorphic change of variables defined in a neighbourhood of infinity which transforms the action $z\mapsto f(z)$ to the multiplicative action $z\mapsto z^{\deg(f)}$.  The relation between this construction and the B\"{o}ttcher coordinate in complex polynomial dynamics is similar to the relation between the complex uniformization of elliptic curves, and Tate's $p$-adic uniformization.   Specifically,  this biholomorphism is Galois equivariant, reducing certain questions about the Galois theory of preimages by $f$ to questions about multiplicative Kummer theory.
\end{abstract}

\maketitle

A well known result of Tate asserts that if $K$ is a complete local field, and $E/K$ is an elliptic curve with split multiplicative reduction, then there is a $q\in K^*$ and a short exact sequence
\[0\longrightarrow q^{\ZZ}\longrightarrow \overline{K}^*\longrightarrow E(\overline{K})\longrightarrow 0,\]
defined by power series.  The action of the absolute Galois  group commutes with the map $\overline{K}^*\to E(\overline{K})$, and  one can use this uniformization shed some light on the $\ell$-adic Galois representation \[\rho_\ell:\Gal(\overline{K}/K)\to  \Aut(T_\ell(E))\cong\GL_2(\ZZ_\ell),\]
which takes an element of the Galois group to its action on the Tate module $T_\ell(E)$.  In particular, given that $E/K$ has split multiplicative reduction, one can show that the image of the Galois representation is reasonably large (see \cite[\S V.6]{ataec}), a first step in the direction of Serre's celebrated open-image theorem.  
The main idea is that the field obtained by adjoining to $K$ the $n$-torsion of $E$ is simply the field $K(q^{1/n}, \zeta)$, where $\zeta$ is a primitive $n$th root of unity; the uniformization turns the problem into one of (multiplicative) Kummer theory, which is well understood.

The purpose of this note is to present a similar uniformization for polynomial dynamical systems over local fields, and to explore the consequences for the action of Galois on preimages of a point. 
Boston and Jones \cite{bostonjones} discuss the action of Galois on the iterated preimages of a point under a polynomial dynamical system in terms of an arboreal Galois representation.  In particular, if $T$ is the infinite, rooted $d$-ary tree associated to preimages of $P$ under a polynomial $f(z)$ of degree $d$, over a field $K$ of characteristic zero, then there is a natural homomorphism
\[\rho_{f, P}:\Gal(\overline{K}/K)\to \Aut(T),\]
called the \emph{arboreal Galois representation} associated to $f(z)$ and $P$.  This representation is analogous to the Galois representation associated to the $\ell$-adic Tate module of an  elliptic curve alluded to above, and our first result shows that one may similarly demonstrate that the images of certain arboreal Galois representations over finite extensions of $\QQ_p$ are large, in the sense of having finite index in a certain type of subgroup of $\Aut(T)$.  

By a \emph{consistent labelling} of $T$ we mean a labelling of the nodes $\alpha_k$ on the $N$th level of $T$ by elements $k\in\ZZ/d^N\ZZ$, for each $N$, in such a way that the node below $\alpha_k$ is labeled by the image of $k$ under the natural map $\ZZ/d^N\ZZ\to \ZZ/d^{N-1}\ZZ$.   By a \emph{Kummer subgroup} of $\Aut(T)$, we mean a subgroup $G$ for which there is a consistent labelling of $T$ such that the restriction of $G$ to the nodes at level $N$ is isomorphic to $\left(\ZZ/d^N\ZZ\right)\rtimes \left(\ZZ/d^N\ZZ\right)^\times$, acting by $(i, j)\alpha_k=\alpha_{jk+i}$.  In other words, a Kummer subgroup of $\Aut(T)$ is one that acts as one might expect Galois to act on the tree of iterated $d$th roots of an element of $K$.

 \begin{theorem}\label{th:finite index}
If $K/\QQ_p$ is a finite extension, if $f(z)\in K[z]$ is a monic polynomial with good reduction and degree not divisible by $p$, and if $P\in K$ is not in the filled Julia set of $f$, then the image of the arboreal Galois representation $\rho_{f, P}$ has finite index in a Kummer subgroup of $\Aut(T)$.  
\end{theorem}

We recall here that the \emph{filled Julia set} of $f$ is simply the set of elements whose forward orbit under $f$ is a bounded subset of $K$.

 The construction which uniformizes the polynomial dynamical system is the $p$-adic analogue to the B\"{o}ttcher coordinate used in complex dynamics (see also a similar construction used by Ghioca and Tucker \cite{ghioca-tucker}).  Throughout, we restrict attention to monic polynomials; it is not difficult to generalize these results to arbitrary polynomials, although in the interest of simplicity, we have avoided doing so here.  Buff, Epstein, and Koch \cite{bek} have extended the notion of a B\"{o}ttcher coordinate to higher-dimensional dynamics; while we have not pursued this here, it is possible that a similar construction over local fields would allow one to say something about the action of Galois on preimages in that context.
 
 Although Theorem~\ref{th:finite index} applies, as stated, only to finite extensions of $\QQ_p$, the general construction is of interest in other settings.
\begin{theorem}\label{th:main}
Let $K$ be a field complete with respect to a non-trivial non-archimedean valuation, let $f(z)\in K[z]$ be a monic polynomial, and suppose that the degree of $f$ is not divisible by the residual characteristic of $K$.  Then  there exists a disk $\basin\subseteq\PP^1(\overline{K})$ about $\infty$, and a power series $\Omega\in K\llbracket z^{-1}\rrbracket$ convergent on $\basin$, such that $\Omega(f(z))=(\Omega(z))^d$ for all $z\in\basin$.  Furthermore, the map $\Omega$ is Galois-equivariant, in the sense that for any $z\in\basin$ and $\sigma\in\Gal(\overline{K}/K)$,  we have $\Omega(z^\sigma)=\Omega(z)^\sigma$.
\end{theorem}

We note that the Galois-equivariance is true for the complex B\"{o}ttcher coordinate as well, but is markedly less interesting, since the only finite extension of $\RR$ is $\CC$.

Recall that a monic polynomial $f(z)=\sum_{i=0}^d a_iz^i\in K[z]$ has good reduction just in case $|a_i|\leq 1$ for all $i$.
It turns out that if $f$ has good reduction, then the set $\basin$ from Theorem~\ref{th:main} is simply the complement of the filled Julia set of $f$, i.e., $\basin$ is the set of $z\in\PP^1(\overline{K})$ such that $f^N(z)\to\infty$ as $N\to\infty$.
This is what allows us, in Theorem~\ref{th:main}, to describe the action of the absolute Galois group on preimages of certain points by polynomial maps with good reduction.  We note that the following consequence for number fields follows immediately.

\begin{corollary}\label{cor:big galois}
Let $K$ be a number field, let $f(z)\in K[z]$ be monic, and suppose that the sequence $P_n\in\overline{K}$ satisfies $f(P_{n+1})=P_n$ for all $n$.
If there exists a prime $\pf$ of $K$, above a rational prime $p\nmid\deg(f)$, 
such that $f$ has good reduction at $\pf$, and $P_0$ is not in the $\pf$-adic filled Julia set, then for all sufficiently large $n$, we have
$[K(P_{n+1}):K(P_n)]=d$.
\end{corollary}

Of course, the result above also applies when $K$ is the function field of a curve over an algebraically closed field of characteristic 0, a fact which we shall exploit below.

Theorem~\ref{th:main} offers us some insight into the action of Galois on preimages of certain points, especially for polynomials with good reduction.  There are several potential applications of this, and we point out two.
First of all, Theorem~\ref{th:main} gives a sufficient condition for a conjecture of Sookdeo on integer points in backwards orbits.  We remind the reader that if $S$ is a set of primes of a number field $K$, then a point $Q\in\PP^1(\overline{K})$ is \emph{$S$-integral} with respect to $P\in\PP^1(K)$ if there is no prime $\pf$ of $K(Q)$ over a prime in $S$, such that the images of $P$ and $Q$ modulo $\pf$ coincide.  Sookdeo conjectures \cite[Conjecture~1.2]{sookdeo} that if $Q$ is not preperiodic for $f\in K(z)$, then any $P\in \PP^1(K)$ has at most finitely many preimages in $\PP^1(\overline{K})$ that are integral with respect to $Q$.
\begin{corollary}\label{cor:sookdeo}
Let $K$ be a number field, and suppose that $f(z)\in K[z]$ is monic, and has at least one place of good reduction lying above a prime not dividing $\deg(f)$ at which $P\in K$ is not in the local filled Julia set.  Then Sookdeo's backward orbit conjecture holds for $f$ and $P$.
\end{corollary}
Indeed, we can prove something stronger.  Sookdeo proves his conjecture under the hypothesis that a \emph{dynamical Lehmer Conjecture} holds for the preimages of $P$, that is, given that
\[\hat{h}_f(Q)\geq \frac{\epsilon}{[K(Q):K]}\]
for all iterated preimages $Q$ of $P$, where $\epsilon>0$ is absolute (although may depend on $f$ and $P$) and $\hat{h}_f$ is the canonical height associated to $f$.  This is equivalent to obtaining a lower bound of the form $[K(Q):K]\gg d^{N}$, for solutions to $f^N(Q)=P$, and by Corollary~\ref{cor:big galois} such a bound holds under the hypotheses of Corollary~\ref{cor:sookdeo}.

Another application concerns preimages in parametrized families of polynomial dynamical systems.
Faber, Hutz, the author, Jones, Manes, Tucker, and Zieve \cite{7auth} considered the problem of how many rational preimages a given rational $a\in\QQ$ could have under quadratic polynomials $f(z)=z^2+c$, as $c\in\QQ$ varies, and obtained a uniform bound for most $a$. Further progress on this problem was recently reported by  Levin \cite{aaron}.  The argument, which in fact proved something much stronger, centred on a geometric analysis of the \emph{$N$th preimage curve}, defined by the polynomial equation $f^N(z)=a$, in the two variables $z$ and $c$.  Given a number field $K$, and a curve $C/K$, let $f(z)\in K(C)[z]$ and $P\in K(C)$.  One prerequisite to generalizing the results in \cite{7auth} to a study of preimages of $P_t$ under $f_t(z)$ is to understand the geometry of the corresponding curves $X^{\mathrm{Pre}}_{f, P}(N)$, which we define as smooth projective models of \[\left\{f_t^N(z)=P_t\right\}\subseteq \AA^1\times C.\]  At the very least, one would hope to show that, generally, the number of components of $X^{\mathrm{Pre}}_{f, P}(N)$ does not increase without bound as $N\to\infty$.
 Since the goal is a geometric result, we replace the number field by an algebraic closure.
\begin{corollary}\label{cor:preim curve}
Let $F=k(X)$ be the function field of some curve over an algebraically closed field of characteristic 0, let $f(z)\in F[z]$, and $P\in F$.  If there exists a place of good reduction for $f$ at which $P$ is not in the filled Julia set, then the number of components of the preimage curve $X^{\mathrm{Pre}}_{f, P}(N)$ is eventually constant (as $N\to\infty$).
\end{corollary}
Pushing this slightly further, with a stronger hypothesis, one can show that if there are at least five distinct places of good reduction for $f$ at which $P$ has a pole of order prime to $d$ (and consequently is not in the filled Julia set), then $X^{\mathrm{Pre}}_{f, P}(1)$ is an irreducible curve of general type, and hence admits only finitely many rational points over any number field.  As a corollary, there will be at most finitely many $t\in C(K)$ such that $P_t$ has a $K$-rational preimage under $f_t(z)$, and a uniform bound then follows from a simple height argument.

\section{Proof of Theorem~\ref{th:main}}

Let $K$ be a  non-archimedean field as in the statement of Theorem~\ref{th:main}.  First we construct a formal series $\Omega(z^{-1})\in K\llbracket z^{-1}\rrbracket$ with the appropriate properties.  Note that this works for an arbitrary field; we do not use the valuation on $K$ here.
\begin{lemma}
Let $f(z)=z^d+\cdots+a_0\in K[z]$, and let $R=K\llbracket z^{-1}\rrbracket$, with maximal ideal $\mf=z^{-1}R$.  Then there exists a series $\Omega\in \mf\setminus\mf^2$ such that
\[\Omega=\lim_{N\to\infty}\left(f^N(z)\right)^{-1/d^N}\]
in the $\mf$-adic topology, where roots are chosen such that \[(f^N(z))^{-1/d^N}=z^{-1}+O(z^{-2}).\]
\end{lemma}

\begin{proof}
Let
$\beta_N=f^N(z)/z^{d^N}$,
so that $\beta_N=1+O(z^{-1})\in R^*$.  Since $1^{d^N}=1$, and since $R$ is Henselian, there exists a $\xi_N=1+O(z^{-1})\in R^*$ such that $\xi_N^{d^N}=\beta^N$.  First, we claim that the sequence $\{\xi_N\}_{N\geq 0}$ converges $\mf$-adically.  After possibly tensoring with $\overline{K}$, we will assume for simplicity that $K$ contains all $d^N$th roots of unity.  Now, since $\xi_N\equiv \xi_{N+1}\equiv 1\MOD{\mf}$, we have
\begin{eqnarray*}
\left|\xi_N-\xi_{N+1}\right|_\mf&=&\prod_{\zeta^{d^{N+1}}=1}\left|\xi_N-\zeta\xi_{N+1}\right|_\mf\\
&=&\left|\xi_N^{d^{N+1}}-\xi_{N+1}^{d^{N+1}}\right|_\mf\\
&=&\left|\beta_N^d-\beta_{N+1}\right|_\mf\\
&=&|z|_\mf^{-d^{N+1}}\left|(f^N(z))^d-f^{N+1}(z)\right|_\mf\\
&=&|z|_\mf^{-d^{N+1}}\left|a_{d-1}(f^N(z))^{d-1}+\cdots+a_0\right|_\mf\\
&=&|z|_\mf^{-d^{N+1}}|z|_\mf^{(d-1)d^{N}}=e^{-d^N}.
\end{eqnarray*}
A simple telescoping sum argument now shows that the sequence $\xi_N\in R$ is $\mf$-adically Cauchy, and so has a limit $\Xi=1+O(z^{-1})\in R^*$.  Now, note that 
\[\left(z^{-1}\xi_N^{-1}\right)^{-d^N}=f^N(z),\]
so that $z^{-1}\xi_N^{-1}=z^{-1}+O(z^{-2})$ is our specified choice of $d^N$th root, 
and
\[z^{-1}\xi_N^{-1}\to z^{-1}\Xi^{-1};\]
we call the latter series $\Omega$.  From the construction it is clear that \[\Omega=z^{-1}+O(z^{-2})\in\mf\setminus\mf^2,\] and that
\[\Omega=\lim_{N\to\infty}\left(f^N(z)\right)^{-1/d^N},\]
where the $d^N$th root is chosen with linear coefficient 1.
\end{proof}

Note that the series $\Omega$ automatically satisfies the functional equation 
\[\Omega\circ f(z)=\Omega^d.\]  It is also worth noting that, since $\Omega=z^{-1}+O(z^{-2})$, there is a formal power series $\Omega^{-1}=z^{-1}+O(z^{-2})$ such that \[\Omega\circ\Omega^{-1}=\Omega^{-1}\circ\Omega=z,\]
defined by the Lagrange inversion formula.   It is not yet clear that these series have positive radius of convergence.  In order to prove this, we will first give two characterizations of the quantity which turns out to be the radius of convergence.

\begin{lemma}
Let $f(z)=z^d+a_{d-1}z^{d-1}+\cdots+a_0\in K[z]$, and let \[C_f=\max_{0\leq i<d}\left\{1, |a_i|^{1/(d-i)}\right\}.\]  Then 
\[C_f^{-1}=\sup\left\{0<\delta\leq 1:|f(z)|=|z|^d\text{ for all }z\in D(\infty; \delta)\right\},\]
where $D(\infty; \delta)=\{z\in\overline{K}:|z|>\delta^{-1}\}$.
\end{lemma}

\begin{proof}
It follows from the ultrametric inequality that $|f(z)|=|z|^d$ for all $z\in D(\infty; C_f^{-1})$,
since $|z^d|>|a_iz^i|$ for $|z|>C_f$ and $0\leq i<d$.
So if $B$ is the supremum defined above, we clearly have $C_f^{-1}\leq B\leq 1$.   If $C_f=1$, then there is nothing to show, so suppose that $C_f^{-1}<1$, implying $|a_i|>1$ for some $i$.  It suffices to show that there exists a $z$ with $|z|=C_{f}^{-1}$, but $|f(z)|\neq |z|^d$. If $i$ is the least index maximizing $|a_i|^{1/(d-i)}$, then the Newton polygon of $f(z)/z^d$ (as a polynomial in $z^{-1}$) contains a line segment joining $(0, 0)$ to $(d-i, v(a_i))$.  Necessarily this polynomial has a root of absolute value $|a_i|^{1/(d-i)}=C_f^{-1}$.  In other words, there exists a $z$ with $|z|=C_f\neq 0$ and $|f(z)|=0$.  It follows at once that $B\leq C_f^{-1}$.
\end{proof}

Our main interest in the more complicated description of $C_f$, beyond the fact that it seems somewhat more fundamental than the simpler definition, is that it immediately implies that $C_{f^N}\leq C_f$ for all $N$, a fact which is somewhat awkward to prove directly.



\begin{lemma}
The series $\Omega$  converges pointwise on $D(\infty; C_{f}^{-1})$, and $\Omega^{-1}$ converges pointwise on $D(0; C_f^{-1}$).
\end{lemma}

\begin{proof}
In the case where $C_f=1$, this is not particularly difficult.  In particular, if $\mathcal{O}\subseteq K$ is the ring of integral elements, then the condition $C_f=1$ implies $\beta_N=f^N(z)/z^{d^N}\in \mathcal{O}[z^{-1}]$ for all $N$.  It follows that $\xi_N(z)=1+O(z^{-1})\in \mathcal{O}\llbracket z^{-1}\rrbracket$ (which is where we use the hypothesis that $\deg(f)$ is a unit in $K$), and so $\Omega\in \mathcal{O}\llbracket z^{-1}\rrbracket$.  But elements of $\mathcal{O}\llbracket z^{-1}\rrbracket$ are convergent on $D(\infty; 1)$.  It also follows from this, and the fact that $\Omega(z)=z^{-1}+O(z^{-2})$, that  we have $|\Omega(z)|=|z|^{-1}$ for all $z\in D(\infty; 1)$.

In general, note that for any $\alpha\in\overline{K}$, the coefficient of $z^{i-d}$ in $f(\alpha z)/(\alpha z)^d$ is precisely $a_{i}\alpha^{i-d}$, which has modulus at most 1 if $|\alpha|= C_f$.  Since the same is true for $f^N(\alpha z)/(\alpha z)^{d^N}$, because $C_f\geq C_{f^N}$, we see that $\xi_N(\alpha z)\in \overline{K}\llbracket z^{-1}\rrbracket$ has integral coefficients, and thus so too does \[z^{-1}\xi_N(\alpha z)^{-1}=\alpha\Omega(\alpha z)=z^{-1}+a_2\alpha z^{-2}+\cdots\] (here $w=z^{-1}$).  An examination of the Lagrange inversion formula shows that if a power series with linear coefficient 1 has coefficients in a given ring, then so does its inverse, and since the inverse of $\alpha\Omega((\alpha^{-1} z)^{-1})$ is $\alpha\Omega^{-1}((\alpha^{-1} z)^{-1})$, the latter series also has integral coefficients.  It follows that both of these series converge for $|w|<1$, and so the series converge for $|w|<|\alpha|^{-1}=C_f^{-1}$.
\end{proof}

It follows from the above that $\Omega(z)$ is an element of the Tate algebra of $D=D(\infty; \epsilon)$, for any $\epsilon<C_f^{-1}$, but we have not actually shown that $\Omega$ is the limit of $(f^N(z))^{-1/d^N}$ in the corresponding norm $\|\cdot\|_D$ (i.e., the uniform norm corresponding to this disk).  If $\CC_v\supseteq K$ is any complete algebraically closed field, let $\epsilon<\epsilon'<C_f^{-1}$, and let $\|\cdot\|_{D'}$ be the supremum norm on the Tate algebra of $D'=D(\infty; \epsilon')$.  Since these norms are multiplicative, we have for any $g\in T_{D'}$ and $z\in D(\infty; \epsilon)$,
\begin{eqnarray*}
|g(z)|&=&\left|z^{\ord_\infty(g)}\right|\cdot\left|z^{-\ord_\infty(g)}g(z)\right|\\
&\leq& \left|z^{\ord_\infty(g)}\right|\cdot\left\|z^{-\ord_\infty(g)}g(z)\right\|_{D'}\\
&\leq&\left(\frac{|z|}{\epsilon'}\right)^{\ord_\infty(g)} \left\|g\right\|_{D'},
\end{eqnarray*}
and so
\[\|g\|_{D}\leq \left(\frac{\epsilon}{\epsilon'}\right)^{\ord_\infty(g)} \left\|g\right\|_{D'}\]
In particular, any sequence of functions $g_N\in T_{D'}$ such that $\|g_N\|_{D'}$ is bounded, and $\ord_\infty(g_N)\to \infty$ will converge 
as a sequence of elements in $T_D$.  Since $\epsilon<\epsilon'<C_f^{-1}$ were arbitrary, and since $(f^N(z))^{-1/d^N}\to \Omega$ in the $\mf$-adic topology, we have $(f^N(z))^{-1/d^N}\to \Omega$ in the Tate algebra $T_D$ for any proper subdisk $D\subseteq D(\infty; C_f^{-1})$, after noting that $\|(f^N(z))^{-1/d^N}\|_{D'}=\epsilon$.

We have now constructed the power series $\Omega$ and $\Omega^{-1}$, and shown that they converge on some disks of positive radii about $\infty$ and $0$, respectively.  If $f$ has good reduction, then these radii are both 1.
To finish the proof of Theorem~\ref{th:main}, it suffices to observe that for any convergent series $\Sigma a_i z^i$ with $a_i\in K$, and any Galois extension $L/K$, we have $\left(\Sigma a_i z^i\right)^{\sigma}=\Sigma a_i(z^\sigma)^i$ for all $z\in L$ and $\sigma\in \Gal(L/K)$.  This follows simply because $\sigma$ will fix the (unique) maximal ideal of $\mathcal{O}_L\subseteq L$, and hence will act continuously on $L$.

\section{Proofs of the other results}

With the proof of Theorem~\ref{th:main} complete, the other results become relatively straightforward.

\begin{proof}[Proof of Theorem~\ref{th:finite index}]  Let $f(z)\in K[z]$ be a monic polynomial of degree indivisible by the residual characteristic of $f$, and suppose that $f$ has good reduction.  We further suppose that $P\in K$ is not in the filled Julia set of $f$.  Then by Theorem~\ref{th:main}, there is a biholomorphic Galois-equivarient mapping $\Omega:D(\infty; 1)\to D(0; 1)$ such that $\Omega(f(z))=(\Omega(z))^d$, for all $z$.  Note that if $Q\in\overline{K}$ satisfies $f^N(Q)=P$, then $Q\in D(\infty; 1)$, and so the domain of $\Omega$ contains the entire preimage tree $T_{f, P}$.  Thus, $\Omega$ induces a tree isomorphism $T_{f, P}\cong T_{z^d, \Omega(P)}$ which respects the action of $\Gal(\overline{K}/K)$.  It follows that the image of the arboreal Galois representation $\rho_{f, P}$ is isomorphic to the image of the representation $\rho_{z^d, \Omega(P)}$.  Since $\Omega(P)$ is not a unit in $K$, standard Kummer theory shows that this image is a finite index subgroup of a Kummer subgroup of $\Aut(T_{z^d, \Omega(P)})$.
\end{proof}

\begin{proof}[Proof of Corollary~\ref{cor:big galois}] Suppose that $f(z)\in K[z]$ has good reduction at $\pf$, a place above a rational prime $p\nmid \deg(f)$, and that $P_0$ is not in the $\pf$-adic filled Julia set.  It is clearly the case that, if $f(P_{n+1})=P_n$ for all $n$, we have $[K(P_{n+1}):K(P_n)]\leq d$, and so it suffices to establish a lower bound
\[[K(P_{n+1}):K(P_n)]\geq [K_\pf(P_{n+1}):K_\pf(P_n)]\geq d,\]
for $n$ sufficiently large.
  But note that, for each $n$, $\Omega(P_{n+1})^d=\Omega(P_n)$.  Also, we have that $|\Omega(P_0)|=|P_0|^{-1}\neq 1$, and so $\Omega(P_0)$ is not a unit.  It follows that, for $n$ sufficiently large, the ramification index of the extension
  \[K_\pf(P_{n+1})/K_\pf(P_n) = K_\pf(\Omega(P_{n+1}))/K_\pf(\Omega(P_{n}))\]
is $d$, giving the lower bound we require.
\end{proof}

\begin{proof}[Proof of Corollary~\ref{cor:preim curve}]
Finally we prove the claims about preimage curves.  Let $K=k(C)$ be the function field of a curve $C$ over an algebraically closed field $k$ of characteristic not dividing $d$, and let $f(z)\in K[z]$ have degree $d$, $P\in K$.  Let  $v$ be a place at which $f$ has good reduction, and such that $P$ is not in the filled $v$-adic Julia set, and $K_v$ be the completion of the local field at $v$.  As above, for any chain $P_n\in \overline{K}$ with $P_0=P$, $f(Pa_{n+1})=\alpha_n$, we have $[K(P_n):K]\geq \epsilon d^{n}$, for some absolute $\epsilon>0$.  In particular, the number of Galois orbits in $f^{-N}(P)$ remains bounded as $N\to \infty$.  Each one of these Galois orbits corresponds to a component of $X^\mathrm{Pre}_{f, P}(N)$.

Note, also, that the argument provided in the proof of Corollary~\ref{cor:big galois} above shows that the ramification index of $v$ in the extension $[K(P_1):K]$ is at least $d/\gcd(v(P), d)$.  In particular, we see that when $P$ has a pole of order prime to $d$ at $v$, $K(P_1)/K$ is an extension of degree $d$ in which $v$ is totally ramified.  In other words, $X^{\mathrm{Pre}}_{f, P}(1)$ is an irreducible curve, and the map $X^{\mathrm{Pre}}_{f, P}(1)\to X^{\mathrm{Pre}}_{f, P}(0)=C$ ramifies completely at the place above $v$.  If there are at least 5 such places, then $X^{\mathrm{Pre}}_{f, P}(1)$ admits a map of degree $d$ to another curve, with ramification divisor having degree at least $5d-5$.  It follows that this curve has genus at least 2.
\end{proof}


\begin{thebibliography}{9}

\bibitem{bostonjones} N.~Boston and R.~Jones.  The image of an arboreal Galois representation. \emph{Pure and Applied Mathematics Quarterly} \textbf{5} no.~1(2009), pp.~213--225.

\bibitem{bek} X.~Buff, A.~Epstein, and S.~Koch.  B\"{o}ttcher coordinates (preprint)

\bibitem{7auth} X.~Faber, B.~Hutz, P.~Ingram, R.~Jones, M.~Manes, T.~J.~Tucker, M.~E.~Zieve, Uniform bounds on pre-images under quadratic dynamical systems. \emph{Math.\ Res.\ Lett.} \textbf{16} (2009),  pp.~87--101.

\bibitem{ghioca-tucker} D.~Ghioca and T.~J.~Tucker. $p$-adic logarithms for polynomial dynamics. (unpublished preprint)

\bibitem{aaron} A.~Levin, Rational preimages in families of dynamical systems. (preprint, 2011)

\bibitem{robert} A.~M.~Robert, \emph{A Course in $p$-adic Analysis}, volume 198 in \emph{Graduate Texts in Mathematics}. Springer, 2000.

\bibitem{ataec} J.~H.~Silverman, \emph{Advanced Topics in the Arithmetic of Elliptic Curves}, volume 151 of \emph{Graduate Texts in Mathematics}, Springer, 1994.

\bibitem{js:ads} J.~H.~Silverman, \emph{The Arithmetic of Dynamical Systems}, volume 241 of \emph{Graduate Texts in Mathematics}, Springer, 2007. 

\bibitem{sookdeo} V.~Sookdeo, Integer points in backwards orbits.  (preprint, 2010)
\end{thebibliography}
\end{document}